\documentclass[12pt, notitlepage]{amsart}
\usepackage{latexsym, amsfonts, amsmath, amssymb, amsthm, cite}
\usepackage{graphicx}
\usepackage[linktocpage=true,colorlinks,citecolor=blue,linkcolor=blue,urlcolor=blue]{hyperref}
\usepackage{amssymb}
\usepackage{cite}
\usepackage{bbm}
\usepackage{amsmath}
\usepackage{latexsym}
\usepackage{amscd}
\usepackage{tikz}
\usepackage{amsthm}
\usepackage{mathrsfs}
\usepackage{url}
\usepackage[utf8]{inputenc}
\usepackage[english]{babel}
\usepackage{amsfonts}
\usepackage{mathtools}
\usepackage{comment}
\usepackage[autostyle]{csquotes}
\usepackage[colorinlistoftodos]{todonotes}
\usepackage[noabbrev,capitalize]{cleveref}
\usepackage{adjustbox}
\usepackage{mathrsfs}
\crefname{equation}{}{}
\usepackage{graphicx, tikz}
\usepackage[margin=1in]{geometry}
\usepackage{microtype}

\numberwithin{equation}{section}
\newcommand\sumstar{\mathop{{\sum\nolimits^{\mathrlap{\star}}}}}
\def\R{\mathbb R}
\def\Z{\mathbb Z}
\def\C{\mathbb C}
\def\P{\mathbb P}

\def\E{\mathbb E}
\def\GK{\mathbf{1}_{\mathcal{G}(k)}}
\def\GKT{\mathbf{1}_{\mathcal{G}(k,t)}}
\def\HKT{\mathbf{1}_{\mathcal{H}(k,t)}}

\def\n{\textbf{n} }

\def\r{\right}

\def\CA{\mathcal{A}}

\def\ee{\varepsilon}

\def\fkr{F_{ k}^{(R)}}

\def\AR{\CA_R(x)}

\newtheorem{theorem}{Theorem}[section]
\newtheorem{lemma}[theorem]{Lemma}
\newtheorem{proposition}[theorem]{Proposition}

\newtheorem{question}[theorem]{Question}

\theoremstyle{remark}
\newtheorem{remark}[theorem]{Remark}

\theoremstyle{definition}

\theoremstyle{remark}

\numberwithin{equation}{section}
\begin{document}
\title[Better than square-root cancellation]{Better than square-root cancellation for random multiplicative functions}

\author{Max Wenqiang Xu}
\address{Department of Mathematics, Stanford University, Stanford, CA, USA}
\email{maxxu@stanford.edu}

\begin{abstract}
    We investigate when the better than square-root cancellation phenomenon exists for $\sum_{n\le N}a(n)f(n)$, where $a(n)\in \mathbb{C}$ and $f(n)$ is a random multiplicative function. We focus on the case where $a(n)$ is the indicator function of $R$ rough numbers. We prove that $\log \log R \asymp (\log \log x)^{\frac{1}{2}}$ is the threshold for the better than square-root cancellation phenomenon to disappear.
\end{abstract}
\maketitle

\section{introduction}

The study of random multiplicative functions has attracted intensive attention. Historically, they were introduced to model arithmetic functions. A Steinhaus random multiplicative function $f(n)$ is a completely multiplicative function defined on positive integers such that $f(p)$ are independently and uniformly distributed on the complex unit circle for all primes $p$. One may view it as a random model for arithmetic functions like Dirichlet characters $\chi(n)$ or $n^{it}$. Another popular model is the Rademacher random multiplicative function $f(n)$ which was first used by Wintner\cite{Win} as a random model for M\"obius function $\mu(n)$. In this note, we focus on the Steinhaus case. The obvious dependence between random variables $f(m)$ and $f(n)$ whenever $(m, n)\neq 1$ makes the study of random multiplicative functions intriguing.

Arguably the most striking result so far in the study of random multiplicative functions is Harper's \cite{HarperLow} remarkable resolution of Helson's conjecture\cite{Helson} (see \cite{BS2016} for some earlier discussions), that is, the partial sums of random multiplicative functions enjoy better than square-root cancellation
\begin{equation}\label{eqn: harper}
  \E [|\sum_{n\le x} f(n) |] \asymp \frac{\sqrt{x}}{(\log \log x)^{1/4}},  
\end{equation}
where $f(n)$ are random multiplicative functions. 
In particular, with the natural normalization $\sqrt{x}$, the partial sums $\sum_{n\le x }f(n)$ do not converge in distribution to the standard complex normal distribution (see also \cite{Harper}). Before Harper's result \cite{HarperLow}, there was progress on proving good lower bounds close to $\sqrt{x}$, e.g. \cite{HNR15}, and it was not clear that such better than square-root cancellation in \eqref{eqn: harper} would appear until Harper's proof. See also recent companion work on analogous results in the character sums and zeta sums cases established by Harper \cite{Harpertypical, harper2019partition}.
It is known that the better than square-root cancellation phenomenon in random multiplicative functions is connected to the ``critical multiplicative chaos" in the probability literature. We point out references \cite{NPS, chaosreview, Ber, DRS, HKO} for related discussions.  

A closely related important question in number theory is to understand the distribution of the Riemann zeta function over typical intervals of length $1$ on the critical line 
$\mathfrak{Re}(s)=\frac12$. One may crudely see the connection by viewing $\zeta(s)$ as a sum of $n^{-\frac12-it}$ for a certain range of $n$ and $n^{it}$ behaves like a Steinhaus random multiplicative function for randomly chosen $t$. A conjecture of Fyodorov, Hiary, and Keating (see e.g. \cite{FK14, FHK12}) suggests that there is a subtle difference between the true order of local maximal of $\log|\zeta(1/2+it)|$ and one's first guess based on Selberg's central limit theorem for $\log|\zeta(1/2+it)|$. The existence of this subtle difference and the appearance of the better than square-root cancellation for random multiplicative functions both show that the corresponding nontrivial dependence can not be ignored. We refer readers to \cite{BK22, AOR21, Gerspach, ABBRS, ABH, harper2019partition, ABR1, ABR2, SoundICM, Halasz, Harpersuprema, Harperlargevalue, LTW13, Najextreme, Gerspachtoy} for related discussions about partial sums of random multiplicative functions and zeta values distribution.

In this paper, we are interested in further exploring Harper's result \eqref{eqn: harper} and methods used there, by considering the problem in a more general context. 
\begin{question}\label{question main} Let $a(n)$ be a sequence in $\C$. 
When does the better than square-root cancellation phenomenon hold for $\sum_{n\le N} a(n)f(n)$, i.e.
\begin{equation}\label{eqn: a_n}
 \E [|\sum_{n\le N}a(n) f(n)|] = o\left(\sqrt{\sum_{n\le N}|a(n)|^{2}}\r)?    
\end{equation}
\end{question}

We first make some simple observations in the situations where $a(n)$ is ``typical" or $a(n)$ has a rich multiplicative structure. Then we focus on a particular case where the coefficient $a(n)$ is an indicator function of a multiplicative set. 
\subsection{Typical coefficients}
If partial sums $\sum_{n\le N}a(n)f(n)$ with the square-root size normalization behave like the complex standard Gaussian variable, then there is just square-root cancellation. One may attempt to prove such a central limit theorem by computing the high moments, however, the moments usually blow up and such a strategy does not work here (see e.g. \cite{wangxu, Harperhigh, HNR15, HL} for moments computation results). It turns out that for ``typical" choices of $a(n)$, such a central limit theorem does hold. It has been carried out in the concrete case where $a(n)=e^{2\pi i n \theta}$ for some fixed real $\theta$ without too good Diophantine approximation properties (such $\theta$ has relative density $1$ in $\R$, e.g. one can take $\theta=\pi$) by Soundararajan and the author \cite{SoundXu}, and also an average version of the result is proved by Benatar, Nishry and Rodgers \cite{BNR}. The proof of the result in \cite{SoundXu} is based on McLeish's martingale central limit theorem\cite{McLeish}, and the method was pioneered by Harper in \cite{Harper}.  
The proof reveals the connection between the existence of such a central limit theorem and a quantity called \textit{multiplicative energy} of $\textbf{a}:= \{a(n): 1\le n\le N\}$ 
\[E_{\times}(\mathbf{a}): = \sum_{\substack{m_1, n_1, m_2, n_2 \le N\\ m_1m_2=n_1n_2} }a(m_1)a(m_2) \overline{a(n_1)a(n_2)}. \]
A special case of $a(n)$ is an indicator function of a set $\CA$, and the quantity $E_{\times}(\CA)$ is a popular object studied in additive combinatorics. It is now known \cite{SoundXu} that a crucial condition for such a central limit theorem to hold for $\sum_{n\le N}a(n)f(n)$ is that the set $\CA$ has multiplicative energy $\le(2+\epsilon)|\CA|^{2}$. See \S \ref{Sec: typical} for more discussions on $a(n)$ being a ``typical" choice.  We refer readers who are interested in seeing more examples of when a central limit theorem holds for partial (restricted) sums of random multiplicative functions to \cite{SoundXu, KSX, Harper, CS, PWX, Hough, BNR}.  

\subsection{Large multiplicative energy and sparse sets} Let us focus on the case that $a_n$ is an indicator function of a set $\CA$. As we mentioned if the set $\CA$ has small multiplicative energy (among other conditions), then partial sums exhibit square-root cancellation. Suppose we purposely choose a set $\CA$ with very large multiplicative energy, will it lead to better than square-root cancellation? One extreme example is $\CA= \{p^{n}: 1\le n \le \log_p N\}$ being a geometric progression, where $p$ is a fixed prime. A standard calculation gives that
\[\E[| \sum_{n\in \CA} f(n) |] =  \int_{0}^{1} |\sum_{n\le \log_p N}e(\theta n)| d\theta \asymp \log \log N,  \]
while $\E[| \sum_{n\in \CA} f(n) |^{2}] 
 = |\CA|\asymp \log N$ and $\E[| \sum_{n\in \CA} f(n) |^{4}] \gg |\CA|^{3}$ which is pretty large. 
It shows that there is a great amount of cancellation in this particular example when the multiplicative energy is large. One may also take $\CA$ to be some generalized (multidimensional) geometric progression and get strong cancellation of this type. We note that the sets mentioned here with very rich multiplicative structures all have small sizes. 

Based on the initial thoughts above, we may lean toward believing that better than square-root cancellation only appears when $a(n)$ has some particular structure that is perhaps related to multiplicativity.  
To fully answer Question~\ref{question main} seems hard. The majority of the paper is devoted to a special case, where $a(n)$ is an indicator function of a set with multiplicative features. We focus on fairly large subsets.
\subsection{Main results: multiplicative support}
Suppose now that $a(n)$ is a multiplicative function with $|a(n)|\le 1$. 
The particular example we study in this paper is that $a(n)$ is the indicator function of $R$-rough numbers, although the proof here may be adapted to other cases when $a(n)$ is multiplicative.  
We write 
\begin{equation}\label{eqn: rough}
    \CA_{R}(x): = \{n\le x: p|n \implies p\ge R\}.
\end{equation}
By a standard sieve argument, for all $2\le R\le x/2$ (the restriction $R\le x/2$ is only needed for the lower bound), we have asymptotically \cite{Bruijn50} 
\begin{equation}\label{eqn: sieve}
   |\CA_R(x)| \asymp \frac{x}{\log R}.  
\end{equation}
We expect the following threshold behavior to happen. If $R$ is very small, the set $\CA_R(x)$ is close to $[1,x]$ and better than square-root cancellation appears as in \cite{HarperLow}. If $R$ is sufficiently large, then weak dependence may even lead to a central limit theorem. Indeed, an extreme case is that $R> \sqrt{x}$, in which $\CA_R(x)$ is a set of primes and $\{f(n): n\in \AR\}$ is a set of independent random variables. It is natural to ask to what extent the appearance of small primes is needed to guarantee better than square-root cancellation.
Our Theorem~\ref{thm: upper} and Theorem~\ref{thm: lower} answer the question. We show that $\log \log R \approx (\log \log x)^{1/2}$ is the threshold.

\begin{theorem}\label{thm: upper}
  Let $f(n)$ be a Steinhaus random multiplicative function and $x$ be large. Let $\CA_R(x)$ be the set of 
  $R$ rough numbers up to $x$. For any $\log \log R\ll (\log \log x)^{\frac{1}{2}}$, we have   
  \[\E[|\sum_{n\in \CA_R(x)} f(n) |] \ll \sqrt{|\CA_{R}(x)|} \cdot  \Big ( \frac{\log \log R +  \log \log \log x}{\sqrt{\log \log x}} \Big)^{\frac{1}{2}} .  \]
 In particular, if $\log \log R = o((\log \log x)^{\frac{1}{2}})$, then  
   \[\E[|\sum_{n\in \CA_R(x)} f(n) |] =o \big( \sqrt{|\CA_{R}(x)|}\big).\]
\end{theorem}
The term $\log \log \log x$ is likely removable. But for the convenience of the proof, we state the above version. See Remark~\ref{rem: 100} for more discussions.

\begin{theorem}\label{thm: lower}
  Let $f(n)$ be a Steinhaus random multiplicative function and $x$ be large. Let $\CA_R(x)$ be the set of 
  $R$ rough numbers up to $x$. For any $\log \log R\gg  (\log \log x)^{\frac{1}{2}}$, we have   
  \[\E[|\sum_{n\in \CA_R(x)} f(n) |] \gg  \sqrt{|\CA_{R}(x)|} .  \]
\end{theorem}
One probably can prove a lower bound of the shape $\sqrt{|\AR|}\cdot(\log \log R / \sqrt{\log \log x})^{-1/2}$
 when $\log \log R =o(\sqrt{\log \log x}$). We do not pursue this as we focus on finding the threshold value of $R$ instead of caring about the quantification of the exact cancellation.

We note that one way to derive a lower bound on $L^{1}$ norm is by proving an upper bound on $L^{4}$ norm. A simple application of H\"older's inequality gives that
\begin{equation}\label{eqn: 4th}
 |\CA_R(x)| = \E[|\sum_{n\in \CA_R(x)} f(n) |^{2}] \le \Big(\E[|\sum_{n\in \CA_R(x)} f(n) |^{4}]\Big)^{1/3} \Big(\E[|\sum_{n\in \CA_R(x)} f(n) |]\Big)^{2/3}.
\end{equation}
The fourth moment  $\ll|\AR|^{2}$ would imply that $L^{1}$ norm $\gg \sqrt{|\AR|}$. However, to achieve such a bound on the fourth moment, one requires $\log R \gg (\log x)^{c}$ for some constant $c$ (by using the method in \cite{SoundXu}), and thus this approach would not give the optimal range as in Theorem~\ref{thm: lower}.

Another reason for studying the fourth moment (multiplicative energy) is to understand the distribution. As mentioned before, this is the key quantity that needs to be understood in order to determine if random sums have Gaussian limiting distribution, via the criteria in \cite{SoundXu}. One may establish a central limit theorem in the range $R\gg \exp((\log x)^{c})$ for some small positive constant $c$.\footnote{One trick to get a smaller $c$ than by directly computing the fourth moment over the full sum is to take the anatomy of integers into account. We refer interested readers to \cite{xuzhou, PS1990} to see how this idea is connected to the correct exponent in extremal sum product conjecture of Elekes and Ruzsa \cite{ElekesRuzsa}.} Interested readers are suggested to adapt the proof of \cite[Corollary 1.2]{SoundXu}. We do not pursue results along this direction in this note. 

Theorem~\ref{thm: upper} and Theorem~\ref{thm: lower} are both proved by adapting Harper's robust method in \cite{HarperLow}, with some modifications, simplifications and new observations, and we sketch the strategy with a focus on how we find the threshold. We also refer readers to a model problem in the function field case by Soundararajan and Zaman \cite{SZ}.  The first step is to reduce the $L^{1}$ norm estimate to a certain average of the square of random Euler products. Basically, we prove that 
\begin{equation}\label{eqn: firststep}
 \E[|\sum_{n\in \AR } f(n)|] \approx \Big(\frac{x}{\log x}\Big)^{1/2} \cdot \E[(\int_{-1/2}^{1/2} |F^{(R)}(\frac{1}{2} + it)|^{2} dt  )^{1/2} ],   
\end{equation}
where $F^{(R)}(1/2+it) := \prod_{R\le p\le x} (1-\frac{f(p)}{p^{1/2+it}})^{-1}$ is the random Euler product over primes $R\le p \le x$. The challenging part is to give a sharp bound on the above expectation involving $|F^{(R)}(1/2+it)|^{2}$ for $|t|\le 1/2$. 

We first discuss the upper bound proof. If we directly apply H\"older's inequality (i.e. moving the expectation inside the integral in \eqref{eqn: firststep}), then 
we would only get the trivial upper bound  $\ll \sqrt{|\AR|}$ as $\E[|F^{(R)}(1/2+it)|^{2}]\approx \log x/\log R$. Harper's method starts with putting some ``barrier events" on the growth rate of all random partial Euler products for all $t$. Roughly speaking, it requires that for all $k$,
\begin{equation}\label{eqn: barieer}
\prod_{x^{e^{-(k+1)}}\le p \le x^{e^{-k}}} |1-\frac{f(p)}{p^{1/2+it}}|^{-1}~\text{``grows as expected" for all $|t|\le 1$}.    
\end{equation}
Denote such a good event by $\mathcal{G}$ and write $s=1/2+it$. By splitting the probability space based on the event $\mathcal{G}$ holding or not, and applying Cauchy–Schwarz inequality, we have 
\[
\begin{split}
\E[(\int_{-1/2}^{1/2} |F^{(R)}(s)|^{2} dt  )^{1/2} ]
&\approx \E[(\int_{-1/2}^{1/2}\mathbf{1}_{\mathcal{G}} |F^{(R)}(s)|^{2} dt  )^{1/2}] +  \E[(\int_{-1/2}^{1/2}\mathbf{1}_{\mathcal{G}~\text{fail}} |F^{(R)}(s)|^{2} dt  )^{1/2}]  \\
& \ll \E[(\int_{-1/2}^{1/2}\mathbf{1}_{\mathcal{G}} |F^{(R)}(s)|^{2} dt  )^{1/2}] + \P(\mathbf{1}_{\mathcal{G}~\text{fail}})^{1/2} (\E[|F^{(R)}(s)|^{2}])^{1/2} .    
\end{split}
 \]
According to the two terms above, there are two tasks that remain to be done. 
\begin{enumerate}
    \item 
Task 1: Show that the expectation is small, conditioning on $\mathbf{1}_{\mathcal{G}}$.
\item Task 2: Show that $\P(\mathbf{1}_{\mathcal{G}~\text{fail}})$ is sufficiently small. 
\end{enumerate} 
To accomplish task 1, Harper's method connects such an estimate to the ``ballot problem" or say Gaussian random walks (see \S \ref{sec: ballot}), which is used to estimate the probability of partial sums of independent Gaussian variables having a certain barrier in growth. Task 2 of estimating the probability of such good events $\mathcal{G}$ happening can be done by using some concentration inequality, e.g. Chebyshev's inequality. 
Our main innovation lies in setting up ``barrier events" in \eqref{eqn: barieer} properly which is not the same as in \cite{HarperLow}. On one hand, it should give a strong enough restriction on the growth rate of the products so that $\E[(\int_{-1/2}^{1/2}\mathbf{1}_{\mathcal{G}} |F^{(R)}(s)|^{2} dt  )^{1/2}]$ has a saving, compared to it without conditioning on $\mathbf{1}_{\mathcal{G}}$. On the other hand, one needs to show that such an event $\mathcal{G}$ is indeed very likely to happen which requires that the designed ``barrier" can not be too restrictive. To achieve the two goals simultaneously, we need $\log \log R = o( \sqrt{\log \log x})$ and this is the limit that we can push to (see Remark~\ref{rem: why}).

The lower bound proof in Theorem~\ref{thm: lower} uses the same strategy as in \cite{HarperLow} but is technically simpler. After the deduction step of reducing the problem to studying a certain average of the square of random Euler products (see \eqref{eqn: firststep}), we only need to give a lower bound of the shape $\gg (\log x / \log R)^{1/2}$ for the expectation on the right-hand side of \eqref{eqn: firststep}. Since the integrand $|F^{(R)}(s)|^{2}$ is positive, it suffices to prove such a lower bound when $t$ is restricted to a random subset $\mathcal{L}$. We choose $\mathcal{L}$ to be the set of $t$ such that certain properly chosen ``barrier events" hold.  The main difficulty is to give a strong upper bound on the restricted product $\E[\mathbf{1}_{t_1, t_2\in \mathcal{L}}|F^{(R)}(1/2+it_1)|^{2}|F^{(R)}(1/2+it_2)|^{2}]$ in the sense that the bound is as effective as in the ideal situation where the factors $|F^{(R)}(1/2+it_1)|^{2}$ and $|F^{(R)}(1/2+it_2)|^{2}$ are independent (see Proposition~\ref{prop: correlation}), and this is also the main reason that the condition $\log \log R \gg \sqrt{\log \log x}$ is needed subject to our chosen ``barrier events". Our proof of Theorem~\ref{thm: lower} does not involve the ``two-dimensional Girsanov calculation", which hopefully makes it easier for readers to follow.

\subsection*{Organization} We set up the proof outline of Theorem~\ref{thm: upper} in Section~\ref{S: thmupper} and defer the proof of two propositions to Section~\ref{S: uppereuler} and Section~\ref{s: upperpro} respectively. We put all probabilistic preparations in Section~\ref{S: prep} which will be used in the proof for both theorems. The proof of Theorem~\ref{thm: lower} is done in Section~\ref{s: lowerthm} and again we defer proofs of two key propositions to Section~\ref{s: lowereuler} and Section~\ref{s: lowerprop} respectively. Finally, we give more details about the ``typical" choices of $a(n)$ in Section~\ref{s: conclude}, as well as mentioning some natural follow-up open problems. 

\subsection*{Acknowledgement}
I would like to thank Adam Harper for helpful discussions, corrections, and comments on earlier versions of the paper and for his encouragement, and thank Kannan Soundararajan for discussions around the topic.  I am indebted to the anonymous referees for many helpful comments and corrections. 
The author is supported by the Cuthbert C. Hurd Graduate Fellowship in the Mathematical Sciences, Stanford. 

\section{Proof of Theorem~\ref{thm: upper}}\label{S: thmupper}
We follow the proof strategy of Harper in \cite{HarperLow}.
We establish Theorem~\ref{thm: upper} in a stronger form that for $1/2\le q \le 9/10$ and $R$ in the given range $\log \log R \ll (\log \log x)^{1/2}$,  
\[ \E [|\sum_{n\in \CA_R(x)} f(n) |^{2q}] \ll |\CA_R(x)|^{q} \Big ( \frac{\log \log R +  \log \log \log x}{\sqrt{\log \log x} } \Big)^{q}.  \]
One should be able to push the range of $q$ to 1 but for simplicity in notation, we omit it. Our interest is really about the case $q=1/2$.
Note that in the given range of $R$, by \eqref{eqn: sieve}, it is the same as proving 
\[ \E [|\sum_{n\in \CA_R(x)} f(n) |^{2q}] \ll \Big(\frac{x}{\log R}\Big)^{q} \Big ( \frac{\log \log R +  \log \log \log x}{\sqrt{\log \log x} } \Big)^{q}.\]

The first step (Proposition~\ref{Prop: upper L2}) is to connect the $L^{1}$ norm of the random sums to a certain average of the square of random Euler products. We define for all $s$ with $\mathfrak{Re}(s)>0$ and integers $0\le k\le \log \log x -\log \log R$, the random Euler products
\begin{equation}\label{eqn: Fk}
 F_{ k}^{(R)}(s) : = \prod_{R\le p\le x^{e^{-(k+1)}}} (1- \frac{f(p)}{p^{s}})^{-1} = \sum_{\substack{n\ge 1 \\ p|n\implies R\le  p\le x^{e^{-(k+1)}}}} \frac{f(n)}{n^{s}}.    
\end{equation}
We also write 
\begin{equation}\label{eqn: F}
    F^{(R)}(s): = \prod_{R\le p\le x} (1- \frac{f(p)}{p^{s}})^{-1} = \sum_{\substack{n\ge 1 \\ p|n\implies R\le  p\le x}} \frac{f(n)}{n^{s}}.  
\end{equation}
We use the notation $\|X\|_{2q}: = (\E[|X|^{2q}])^{\frac{1}{2q}}$ for a random variable $X$. 
\begin{proposition}\label{Prop: upper L2}
Let $f(n)$ be a Steinhaus random multiplicative function and $x$ be large. Let $F_k^{(R)}(s)$ be defined as in \eqref{eqn: Fk} and $\log \log R \ll (\log \log x)^{\frac{1}{2}}$. Set $\mathcal{K}: =\lfloor  \log \log \log x \rfloor$. Then uniformly for all $1/2 \le q\le 9/10$, we have 
\begin{equation}\label{eqn: qbound}
 \|\sum_{n\in \AR} f(n)\|_{2q} \le \sqrt{\frac{x}{\log x}} \sum_{0\le k \le \mathcal{K} } \Big\|\int_{-1/2}^{1/2} | F_{ k}^{(R)}(\frac{1}{2} - \frac{k}{\log x}+it)|^{2}dt\Big\|_q^{\frac{1}{2}} + \sqrt{\frac{x}{\log x}}.     
\end{equation}
\end{proposition}
We remind the readers that the upper bound we aim for in Theorem~\ref{thm: upper} is very close to $\sqrt{x/\log R}$. The second term in \eqref{eqn: qbound} is harmless since $\log R$ is much smaller than $\log x$.

The second step deals with the average of the square of random Euler products in \eqref{eqn: qbound}, which lies at the heart of the proof.

\begin{proposition}\label{Prop: mean }
Let $F_k^{(R)}(s)$ be defined as in \eqref{eqn: Fk} and $\log \log R \ll (\log \log x)^{\frac{1}{2}}$. Then for all $0\le k \le \mathcal{K}=\lfloor  \log \log \log x \rfloor$, and uniformly for all $1/2\le q\le 9/10$, we have 
\[
\E\left[\left(\int_{-\frac{1}{2}}^{\frac{1}{2}} |F_{ k}^{(R)}(\frac{1}{2} - \frac{k}{\log x} + it)|^{2}dt\r)^{q}\r] \ll e^{-\frac{k}{2}}\cdot  \left(\frac{\log x}{\log R} \r)^{q} \Big ( \frac{\log \log R + \log \log \log x}{\sqrt{\log \log x}} \Big)^{q} .
\]
\end{proposition}

\bigskip 

\begin{proof}[Proof of Theorem~\ref{thm: upper} assuming Proposition~\ref{Prop: upper L2} and Proposition~\ref{Prop: mean }]
Apply Proposition~\ref{Prop: upper L2} and Proposition~\ref{Prop: mean } with $q=\frac{1}{2}$. Notice that when  $\log \log R \ll (\log \log x)^{1/2}$, the term $\sqrt{\frac{x}{\log x}}$ in \eqref{eqn: qbound} is negligible and we complete the proof. 
\end{proof}

\section{Probabilistic preparations}\label{S: prep}
In this section, we state some probabilistic results that we need to use later. The proof can be found in \cite{HarperLow} (with at most very mild straightforward modification). 
\subsection{Mean square calculation}
We first state results on mean square calculations.

\begin{lemma}\label{lem: mean square}
Let $f$ be a Steinhaus random multiplicative function. Then for any $400<x\le y$ and $\sigma>-1/\log y$, we have
\begin{equation}\label{eqn: meansquare}
    \E[\prod_{x<p\le y }|1-\frac{f(p)}{p^{\frac{1}{2}+\sigma }}|^{-2}] = \exp\Big( \sum_{x<p\le y} \frac{1}{p^{1+2\sigma}} + O(\frac{1}{\sqrt{x}\log x}) \Big).
\end{equation}
\end{lemma}

The proof is basically using the Taylor expansion and the orthogonality deduced from the definition of a Steinhaus random multiplicative function. See \cite[Lemma 1, and (3.1)]{HarperLow}.

We also quote the following result on two-dimensional mean square calculations. This will be used in proving the lower bound in Theorem~\ref{thm: lower}. 
\begin{lemma}\label{lem: 2meansquare}
   Let $f$ be a Steinhaus random multiplicative function. Then for any $400<x\le y$ and $\sigma>-1/\log y$, we have
 \begin{equation}\label{eqn: 4.2}
    \E[\prod_{x<p\le y }|1-\frac{f(p)}{p^{\frac{1}{2}+\sigma }}|^{-2} |1-\frac{f(p)}{p^{\frac{1}{2}+\sigma+it}}|^{-2} ] = \exp\left( \sum_{x<p\le y} \frac{2+2\cos(t\log p)}{p^{1+2\sigma}} + O(\frac{1}{\sqrt{x}\log x}) \r). 
\end{equation}  
Moreover, if $x>e^{1/|t|}$, then we further have 
\begin{equation}\label{eqn: 4.3}
   = \exp\Big( \sum_{x<p\le y} \frac{2}{p^{1+2\sigma}} + O(1) \Big) .  
\end{equation}
\end{lemma}

The proof of \eqref{eqn: 4.2} is in \cite[(6)]{HarperLow}. To deduce \eqref{eqn: 4.3}, we only need to show the contribution involves $\cos(t\log p)$ terms are $\ll 1$, which follows from a strong form of prime number theorem. See how it is done in \cite[Lemma 5]{HarperLow} and \cite[Section 6.1]{Harpersuprema}.

\subsection{Gaussian random walks and the ballot problem}\label{sec: ballot}
A key probabilistic result used in Harper's method is the following (modification of) a classical result about Gaussian random walks, which is connected to the ``ballot problem". 

\begin{lemma}[Probability result 1, \cite{HarperLow}]\label{lem: ballot}
    Let $a \ge 1$. For any integer $n > 1$, let $G_1, \dots , G_n$ be independent real
Gaussian random variables, each having mean zero and variance between $1/
20$ and $20$, say. Let
h be a function such that $|h(j)| \le 10 \log j$. Then
\[\P\Big(\sum_{m=1}^{j} G_m \le a + h(j), \quad \forall 1\le j\le n\Big) \asymp \min\{1, \frac{a}{\sqrt{n}}\}.\]
\end{lemma}
Without the term $h(j)$, it is a classical result and actually that is all we need in this paper. However, we state this stronger form as the $h(j)$ term can be crucial if one wants to remove the $ \log \log \log x$ factor in Theorem~\ref{thm: upper}. We expect the random sum is fluctuating on the order of $\sqrt{j}$ (up to step $j$) and so the above result is expected to be true. The quantity $h(j)$ is much smaller compared to $\sqrt{j}$ so it is negligible in computing the probability. 

We do not directly use the above lemma. We shall use an analogous version for random Euler products (Proposition~\ref{prop 5}). We do the 
Girsanov-type calculation in our study (an analogue of Girsanov’s theorem from
the theory of Gaussian random variables). As in \cite{HarperLow}, we introduce the probability measure (here $x$ is large and $|\sigma|\le 1/100$, say)
\[\tilde{\P}(A) : = \frac{\E[1_A \prod_{p\le x^{1/e}} |1-\frac{f(p)}{p^{\frac{1}{2}+\sigma}}|^{-2}  ]}{\E[\prod_{p\le x^{1/e}} |1-\frac{f(p)}{p^{\frac{1}{2}+\sigma}}|^{-2}]} .\]
For each $\ell \in \mathbb{N} \cup \{0\}$, we denote the $\ell$-th increment of the Euler product 
\begin{equation}\label{eqn: I_l}
   I_{\ell}(s):= \prod_{x^{e^{-(\ell+2)}}<p\le x^{e^{-(\ell+1)}} } (1-\frac{f(p)}{p^{s}})^{-1}.  
\end{equation}
Since we are restricted to $R$-rough numbers $n$,  the parameter $\ell$ lies in the range $0\le \ell\le \log \log x - \log \log R$. All the remaining setup is exactly the same as in \cite{HarperLow}. 

\begin{proposition}\label{prop 5}
There is a large natural number $B$ such that the following is true.
Let $n\le \log \log x - \log \log R - (B+1)$, and define the decreasing sequence $(\ell_j)_{j=1}^{n}$ of non-negative integers by $\ell_j = \lfloor \log \log x -\log \log R \rfloor -(B+1) - j$. Suppose that $|\sigma|\le \frac{1}{e^{B+n+1}}$, and
that $(t_j)_{j=1}^{n}$ is a sequence of real numbers satisfying $|t_j|\le \frac{1}{j^{2/3} e^{B+j+1}}$ for all $j$.

Then uniformly for any large a and any function $h(n)$ satisfying $|h(n)| \le 10 \log n$, and with $I_{\ell}(s)$ defined as in \eqref{eqn: I_l}, we have  
\[\tilde{\P}(-a -Bj \le \sum_{m=1}^{j} \log |I_{\ell_m} (\frac{1}{2}+\sigma + it_m)| \le a + j + h(j), \quad \forall j\le n ) \asymp \min\{1, \frac{a}{\sqrt{n}}\} .\]
\end{proposition}
One may view the above sum approximately as a sum of $j$ independent random variables and each with mean $\approx \sum_{x^{e^{-(\ell+2)}}<p\le x^{e^{-(\ell+1)}} } \frac{1}{p} \approx 1$ and with constant variance between $1/20$ and $20$. This shows the connection to Lemma~\ref{lem: ballot}. The deduction of Proposition~\ref{prop 5} from Lemma~\ref{lem: ballot} can be found in the proof of \cite[Proposition 5]{HarperLow}. The only modification is changing the upper bound restriction from $n \le \log \log x -(B+1)$ to $n \le \log \log x - \log \log R -(B+1)$ and all conditions remaining are
satisfied.

\section{Proof of Proposition~\ref{Prop: upper L2}}\label{S: uppereuler}
The proof follows closely to the proof of \cite[Proposition 1]{HarperLow}. 
For any integer $0\le k \le \mathcal{K}= \lfloor  \log \log \log x \rfloor$, let 
\begin{equation}
  I_k: =(x_{k+1}, x_{k}] :=  (x^{e^{-(k+1)} } , x^{e^{-k}}].  
\end{equation}
Let $P(n)$ be the largest prime factor of $n$. For simplicity, we use $\sumstar_n$ to denote the sum where the variable $n$ is $R$-rough.
By using Minkowski's inequality (as $2q\ge 1$), 
\begin{equation}\label{eqn: triangle}
  \|\sum_{n\in \AR } f(n)\|_{2q} \le
\sum_{0\le k \le \mathcal{K}}\|\sumstar_{\substack{n\le x \\ P(n)\in I_k}} f(n)\|_{2q} + \|\sumstar_{\substack{n\le x\\P(n)\le x^{e^{-(\mathcal{K}+1)}}}} f(n)\|_{2q}
.  
\end{equation}
We first bound the last term by only using the smoothness condition and it is bounded by 
$\le \Psi (x, x^{1/\log \log x}  )^{\frac{1}{2}} \ll \sqrt{x} (\log x)^{-c\log \log \log  x}$, which is acceptable. Here $\Psi(x, y)$ denotes the number of positive integers up to $x$ with no primes bigger than $y$ and the estimate is standard, see \cite{Gran08}.
The main contribution to the upper bound in \eqref{eqn: triangle} can be written as
\[ = \sum_{0\le k \le \mathcal{K}} \| \sum_{\substack{m\le x \\ p|m \implies p \in I_k}}f(m) \sumstar_{\substack{n\le x/m \\ n~\text{is $x_{k+1}$-smooth}}} f(n) \|_{2q}.
\]
We now condition on $f(p)$ for $p$ small but at least $R$. Write $\E^{(k)}$ to denote the expectation conditional on $(f(p))_{p\le x_{k+1}}$. Then the above is
\[
\begin{split}
    & = \sum_{0\le k \le \mathcal{K}} (\E\E^{(k)} [|\sum_{\substack{m\le x\\ p|m \implies p \in I_k}} f(m) \sumstar_{\substack{n\le x/m\\ n~\text{is $x_{k+1}$-smooth}}} f(n)|^{2q}])^{1/2q} \\
    & \le \sum_{0\le k \le \mathcal{K}} (\E[(\E^{(k)} [|\sum_{\substack{m\le x\\ p|m \implies p \in I_k}}f(m) \sumstar_{\substack{n\le x/m\\ n~\text{is $x_{k+1}$-smooth}}} f(n)|^{2}])^{q}])^{1/2q}\\
    & = \sum_{0\le k \le \mathcal{K}} ( \E[( \sum_{\substack{m\le x \\ p|m \implies p\in I_k} }  |\sumstar_{\substack{n\le x/m\\ n~\text{is $x_{k+1}$-smooth} }} f(n)|^{2})^{q}] )^{\frac{1}{2q}}.
\end{split}
\]
Then we only need to show that for each expectation in the sum, it is bounded as in \eqref{eqn: qbound}. We next replace the discrete mean value with a smooth version, i.e. we want to replace the sum with some integral. Set $X=e^{\sqrt{\log x}}$, and we have the expectation involving primes in $I_k$ is 
\begin{equation}\label{eqn: tri}
 \begin{split}
&\ll \E\left[\left(\sum_{\substack{m\le x \\ p|m \implies p\in I_k} } \frac{X}{m} \int_{m}^{m(1+\frac{1}{X})} |\sumstar_{\substack{n\le x/t\\ n~\text{is $x_{k+1}$-smooth} }} f(n)|^{2} dt \r)^{q}\r]  \\
& + \E\left[\left( \sum_{\substack{m\le x \\ p|m \implies p\in I_k} } \frac{X}{m} \int_{m}^{m(1+\frac{1}{X})} |\sumstar_{\substack{x/t \le n\le x/m\\n~\text{is $x_{k+1}$-smooth} }} f(n)|^{2} dt \r)^{q} \r].  
\end{split}   
\end{equation}
By using H\"{o}lder's inequality, we upper bound the second term in \eqref{eqn: tri} by the $q$-th power of
\begin{equation}\label{eqn: q}
    \sum_{\substack{m\le x \\ p|m \implies p\in I_k} } \frac{X}{m} \int_{m}^{m(1+\frac{1}{X})} \E [|\sumstar_{\substack{x/t \le n\le x/m\\n~\text{is $x_{k+1}$-smooth} }} f(n)|^{2}] dt . 
\end{equation}
Do a
mean square calculation (analogous to  \eqref{eqn: meansquare}) and throw away the restriction on the $R$-rough numbers. Then \eqref{eqn: q} is at most $\ll  2^{-e^{k}} x/\log x$ and thus the second term in \eqref{eqn: tri} is $\ll(2^{-e^{k}} x/\log x)^{q}$. Summing over $k\le \mathcal{K}$, this is acceptable and thus we only need to focus on the first term in \eqref{eqn: tri}. By swapping the summation, it is at most
\[\E \left[ \left( \int_{x_{k+1}}^{x} |\sumstar_{\substack{n\le x/t \\ n~\text{is $x_{k+1}$-smooth}  }}f(n) |^{2} \sum_{\substack{t/(1+1/X)\le m \le t\\ p|m \implies p\in I_k  }} \frac{X}{m} dt\r)^{q} \r]. \]
We upper bound the sum over $m$ by using a simple sieve argument (sieving out all primes in $[2, t^{1/10}]\backslash I_k$) to derive that the above is at most 
\[\E\left[ \left ( \int_{x_k}^{x} |\sumstar_{\substack{n\le x/t \\ n~\text{is $x_{k+1}$-smooth}  }}f(n) |^{2} \frac{dt}{\log t}\r)^{q} \r] = x^{q} \E\left[ \left ( \int_{1}^{x/x_{k+1}} |\sumstar_{\substack{n\le z \\ n~\text{is $x_{k+1}$-smooth}  }}f(n)|^{2} \frac{dz}{z^{2}\log(\frac{x}{z})} \r)^{q} \r], \]
where in the equality above we used the substitution $z: =x/t$. 
A simple calculation shows that we can replace $\log(x/z)$ by $\log x$ without much loss. Indeed, if $z\le \sqrt{x}$ then $\log(x/z)\gg \log x$; if $\sqrt{x}\le z \le x/x_{k+1}$ then $\log (x/z) \ge z^{-2k/\log x}\log x$.  Thus, we further have the bound
\begin{equation}\label{eqn: pars}
  \ll \left(\frac{x}{\log x}\r)^{q} \E \left[ \left( \int_{1}^{x/x_{k+1}} |\sumstar_{\substack{n\le z \\ n~\text{is $x_{k+1}$-smooth} }}f(n)|^{2} \frac{dz}{z^{2-2k/\log x}} \r)^{q} \r] .  
\end{equation}
To this end, we apply the following version of Parseval's identity, and its proof can be found in \cite[(5.26) in Sec 5.1]{MV2007}.

\begin{lemma}[{\cite[Harmonic Analysis Result 1]{HarperLow}}]\label{lem: pars}
    Let $(a_n)_{n=1}^{\infty}$ be any sequence of complex numbers, and let $A(s): = \sum_{n=1}^{\infty} \frac{a_n}{n^{s}}$ denote the corresponding Dirichlet series, and $\sigma_c$ denote its abscissa of convergence. Then for any $\sigma> \max\{0, \sigma_c\}$, we have 
    \[\int_{0}^{\infty} \frac{|\sum_{n\le x}a_n|^{2}}{x^{1+2\sigma}}dx = \frac{1}{2\pi} \int_{-\infty}^{+\infty} \Big|\frac{A(\sigma + it)}{\sigma + it}\Big|^{2} dt.  \]
\end{lemma}
Apply Lemma~\ref{lem: pars} and the expectation in \eqref{eqn: pars} is
\[ =
\E\left[ \left (\int_{-\infty}^{+\infty} \frac{|\fkr(\frac{1}{2}-\frac{k}{\log x} +it)|^{2}}{|\frac{1}{2}-\frac{k}{\log x} +it|^{2}} dt\r)^{q} \r] \le \sum_{n\in \Z} \E\left[\left( \int_{n-\frac{1}{2}}^{n+\frac{1}{2}} \frac{|\fkr(\frac{1}{2}-\frac{k}{\log x} +it)|^{2}}{|\frac{1}{2}-\frac{k}{\log x} +it|^{2}} dt \r)^{q}\r].
\]
Since $f(m)m^{it}$ has the same law as $f(m)$ for all $m$, for any fixed $n$ we have
\[\E\left[\left(\int_{n-\frac{1}{2}}^{n+\frac{1}{2}} |\fkr(\frac{1}{2}-\frac{k}{\log x} +it)|^{2} dt \r)^{q}\r] = \E\left[ \left(\int_{-\frac{1}{2}}^{\frac{1}{2}} |\fkr(\frac{1}{2}-\frac{k}{\log x} +it)|^{2} dt \r)^{q}\r]. \]
For $n-1/2\le t\le n+1/2$, we have 
$1/|\frac{1}{2}-\frac{k}{\log x} +it|^{2}\asymp 1/(1+n^{2})$ which is summable over $n$. 
We complete the proof by inserting the above estimates into \eqref{eqn: pars}.

\section{Proof of Proposition~\ref{Prop: mean }}\label{s: upperpro}

This is the key part of the proof that reveals how $\log \log R \approx \sqrt{\log \log x}$ could become the transition range. 
We begin with a discretization process which is the same as in \cite{HarperLow}.  For each $|t|\le \frac{1}{2}$, set $t(-1)=t$, and then iteratively for each $0\le j \le \log(\log x /\log R) -2$ define 
\[t(j): = \max\{u\le t(j-1): u = \frac{n}{((\log x) /e^{j+1})\log ((\log x) /e^{j+1})}\text{~for some $n\in \Z$}\}. \]
By the definition, we have \cite[(4.1)]{HarperLow}
\[|t-t(j)|\le \frac{2}{((\log x) /e^{j+1})\log ((\log x)/e^{j+1})}.\]
Given this notation, let $B$ be the large fixed natural number from Proposition~\ref{prop 5}. Let $\mathcal{G}(k)$ denote the event that for all $|t|\le \frac{1}{2}$ and for all $k\le j \le \log \log x - \log \log R -B -2$, we have
\begin{equation}\label{eqn: Gk}
   (\frac{\log x}{e^{j+1} \log R} e^{C(x)}  )^{-1} \le \prod_{\ell = j }^{\lfloor \log \log x -\log \log R \rfloor-B-2}
|I_{\ell}(\frac12-\frac{k}{\log x} +it(\ell)) | \le \frac{\log x}{e^{j+1}\log R} e^{C(x)}, 
\end{equation}
where notably, our $C(x)$ is chosen as 
\begin{equation}\label{eqn: C(x)}
 C(x):=\log \log R + 100 \log \log \log x. 
\end{equation}
We shall establish the following two key propositions. The first proposition says that when we are restricted to the good event $\mathcal{G}(k)$, the $q$-th moment is small. 
\begin{proposition}\label{key prop 1}
Let $x$ be large and $\log \log R \ll (\log \log x)^{\frac{1}{2}}$. Let 
$C(x)$ be defined as in \eqref{eqn: C(x)}. Let $F_k^{(R)}$ be defined as in \eqref{eqn: Fk} and $\mathcal{G}(k)$ be defined as in \eqref{eqn: Gk}. For all $0\le k \le \mathcal{K} = \lfloor \log \log \log x\rfloor$ and $1/2\le q \le 9/10$, we have 
\[\E\left[\left(\int_{-\frac{1}{2}}^{\frac{1}{2}} \mathbf{1}_{\mathcal{G}(k)} |F_{ k}^{(R)}(\frac{1}{2} - \frac{k}{\log x} + it)|^{2}dt\r)^{q}\r] \ll \left( \frac{\log x }{e^{k}\log R}  \r)^{q}\Big ( \frac{C(x)}{\sqrt{\log \log x}} \Big)^{q}. \]
\end{proposition}
The second proposition is to show that indeed $\mathbf{1}_{\mathcal{G}(k)}$ happens with high probability. 

\begin{proposition}\label{key prop 2}
Let $\mathcal{G}(k)$ be defined as in \eqref{eqn: Gk}. For all $0\le k \le \mathcal{K}= \lfloor \log \log \log x\rfloor$ and uniformly for all $1/2\le q \le 9/10$ and
$C(x)$ defined in \eqref{eqn: C(x)}, we have
\[\P(\mathcal{G}(k)~\text{fails}) \ll e^{-C(x)}.    \] 
\end{proposition}
The above two key propositions imply Proposition~\ref{Prop: mean }. 
\begin{proof}[Deduction of Proposition~\ref{Prop: mean }]
   According to the good event $\mathcal{G}(k)$ happening or not, we have 
   \[\begin{split}
& \E\left[\left(\int_{-\frac{1}{2}}^{\frac{1}{2}}  |F_{ k}^{(R)}(\frac{1}{2} - \frac{k}{\log x} + it)|^{2}dt\r)^{q}\r] \\ & \le \E\left[\left(\int_{-\frac{1}{2}}^{\frac{1}{2}} \mathbf{1}_{\mathcal{G}(k)} |F_{ k}^{(R)}(\frac{1}{2} - \frac{k}{\log x} + it)|^{2}dt\r)^{q}\r] +   \E\left[\left(\int_{-\frac{1}{2}}^{\frac{1}{2}} \mathbf{1}_{\mathcal{G}(k)\text{fails}} |F_{ k}^{(R)}(\frac{1}{2} - \frac{k}{\log x} + it)|^{2}dt\r)^{q}\r] \\
& 
\le \left( \frac{\log x }{e^{k}\log R}  \r)^{q}\Big ( \frac{C(x)}{\sqrt{\log \log x}} \Big)^{q} +   \left(\int_{-\frac{1}{2}}^{\frac{1}{2}} \E [|F_{ k}^{(R)}(\frac{1}{2} - \frac{k}{\log x} + it)|^{2}]dt \r)^{q}\P(\mathcal{G}(k)~\text{fails})^{1-q},
\end{split}
   \]
  where in the first term we used Proposition~\ref{key prop 1} and we applied H\"older's inequality with exponents $\frac{1}{q}, \frac{1}{1-q}$ to get the second term. We next apply the mean square calculation \eqref{eqn: meansquare} to derive that the above is 
    \[\ll \left( \frac{\log x }{e^{k}\log R}  \r)^{q}\left( \Big ( \frac{C(x)}{\sqrt{\log \log x}} \Big)^{q} +  \P(\mathcal{G}(k)~\text{fails})^{1-q} \r).\]
  Plug in the definition of $C(x)$ and use Proposition~\ref{key prop 2} with $1-q \ge 1/10$ (and then the exceptional probability to the power $1/10$ is negligible) to deduce that the above is
  \[\ll   e^{-k/2}\left( \frac{\log x }{\log R}  \r)^{q} \cdot \Big ( \frac{C(x)}{\sqrt{\log \log x}} \Big)^{q},  \]
  which completes the proof. 
\end{proof}

\begin{remark}\label{rem: why}
We remark that in \eqref{eqn: C(x)}, the quantity $C(x)= \log \log R +100 \log \log \log x$ is different from just being a constant $C$ in \cite{HarperLow}. The reason for our choice of $C(x)$ is the following. Firstly, to keep the $q$-th moment in Proposition~\ref{key prop 1} has a saving (i.e. to make $ \Big ( \frac{C(x)}{\sqrt{\log \log x}} \Big)^{q}$ small), we require that $C(x)= o(\sqrt{\log \log x})$.  Secondly, it turns out that in order to make the exceptional probability in Proposition~\ref{key prop 2} small enough, one has the constraint $\log \log R \ll C(x)$. The combination of the above two aspects together leads to $\log \log R =o(\sqrt{\log \log x})$. 
\end{remark}

\begin{remark}\label{rem: 100}
   In the deduction of Proposition~\ref{Prop: mean }, we did not use an iterative process as used in \cite{HarperLow}. Instead, we added an extra term $100\log \log \log x$ for the purpose of getting strong enough bounds on $\P(\mathcal{G}(k)~\text{fails})$. We simplified the proof by getting a slightly weaker upper bound in Theorem~\ref{thm: upper} as compensation.
\end{remark}

\subsection{Proof of Proposition~\ref{key prop 1}}

The proof of Proposition~\ref{key prop 1} is a simple modification of the proof of Key Proposition 1 in \cite{HarperLow}. We emphasize again that we will use $C(x)$ defined in \eqref{eqn: C(x)} instead of just a constant as in \cite{HarperLow}, and we do not need the extra help from the quantity $h(j)$ which hopefully makes the proof conceptually easier.  

By using H\"older's inequality, it suffices to prove that 
\begin{equation}\label{eqn: Good}
   \E [\GK \int_{-\frac{1}{2}}^{\frac{1}{2}} |F_k^{(R)}(\frac{1}{2}-\frac{k}{\log x} + it)|^{2}dt ]\ll e^{-k}\cdot \frac{\log x}{\log R} \cdot  \frac{C(x)}{\sqrt{\log \log x}} ,  
\end{equation}
uniformly for $0\le k \le \mathcal{K} = \lfloor \log \log \log x \rfloor $.
We can upper bound the left-hand side of \eqref{eqn: Good} by 
\begin{equation}\label{eqn: 5.4}
 \le \int_{-\frac{1}{2}}^{\frac{1}{2}} \E[\GKT |F_k^{(R)}(\frac{1}{2}-\frac{k}{\log x} + it)|^{2}] dt   
\end{equation}
where $\GKT$ is the event that, for any given $t$,  
\[(\frac{\log x}{e^{j+1} \log R} e^{C(x)}  )^{-1} \le \prod_{\ell = j }^{\lfloor \log \log x -\log \log R \rfloor-B-2}
|I_{\ell}(\frac{1}{2}-\frac{k}{\log x} +it(\ell)) | \le \frac{\log x}{e^{j+1}\log R} e^{C(x)}
\]
for all $k\le j \le \log \log x -\log \log R -B -2$. This is an upper bound as $\GK$ is the event of $\GKT$ holds for all $|t|\le \frac{1}{2}$. By the fact that $f(n)$ has the same law as $f(n)n^{it}$, we have
\begin{equation}\label{eqn: H}
  \int_{-\frac{1}{2}}^{\frac{1}{2}} \E[\GKT |F_k^{(R)}(\frac{1}{2}-\frac{k}{\log x} + it)|^{2}] dt = \int_{-\frac{1}{2}}^{\frac{1}{2}} \E[\HKT |F_k^{(R)}(\frac{1}{2}-\frac{k}{\log x} )|^{2} ]dt, 
\end{equation}
where $\HKT$ denotes the event that, for any given $t$,  
\[
    (\frac{\log x}{e^{j+1} \log R} e^{C(x)}  )^{-1} \le \prod_{\ell = j }^{\lfloor \log \log x -\log \log R \rfloor-B-2}
|I_{\ell}(\frac{1}{2}-\frac{k}{\log x} +i(t(\ell)-t)) | \le \frac{\log x}{e^{j+1}\log R} e^{C(x)},
\]
for all $k\le j \le \log \log x- \log \log R - B -2$. 
 We next apply Proposition~\ref{prop 5}. 
It is clear that $\mathcal{H}(k,t)$ is the event treated in Proposition~\ref{prop 5} with $n=\lfloor \log \log x- \log \log R \rfloor -(B+1)-k$; $\sigma= \frac{-k}{\log x}$ and $t_m = t(\lfloor \log \log x- \log \log R \rfloor -(B+1) -m)-t$ for all $m$; and 
\[a = C(x) + B+1, \quad h(j)=0. \]
The parameters indeed satisfy $|\sigma|\le \frac{1}{e^{B+n+1}}$ and $|t_m|\le  \frac{1}{m^{2/3}e^{B+m+1}}$
for all $m$. Apply Proposition~\ref{prop 5} to derive
\[\frac{\E[\HKT |F_k^{(R)}(\frac{1}{2}-\frac{k}{\log x})|^{2}]}{\E[|F_k^{(R)}(\frac{1}{2}-\frac{k}{\log x})|^{2}]} = \tilde{\P}(\mathcal{H}(k,t)) \ll \min\{ 1, \frac{a}{\sqrt{n}}\}.
\]
A simple mean square calculation (see \eqref{eqn: meansquare}) gives that 
\[\E[|F_k^{(R)}(\frac{1}{2}-\frac{k}{\log x})|^{2}] = \exp\left(\sum_{R\le p \le x^{e^{-(k+1)}}} \frac{1}{p^{1-2k/\log x}} +O(1)\r)\ll \frac{\log x}{e^{k}\log R}.\]
Combining the above two inequalities and the relation in \eqref{eqn: H}, we get the desired upper bound for the quantity in \eqref{eqn: 5.4}. Thus, we complete the proof of \eqref{eqn: Good} and Proposition~\ref{key prop 1}.

\subsection{Proof of Proposition~\ref{key prop 2}}
In the proof, we will see why it is necessary to make $C(x)$ large enough compared to $\log \log R$. The proof starts with the union bound. We have 
\[\P(\mathcal{G}(k)~\text{fails}) \le \P_1 +\P_2,\]
where
\[\P_1 = \sum_{k\le j \le \log (\frac{\log x}{\log R}) -B-2} \P\left( \prod_{\ell = j }^{\lfloor \log (\frac{\log x}{\log R}) \rfloor-B-2}
|I_{\ell}(\frac{1}{2}-\frac{k}{\log x} + i t(\ell)) | >\frac{\log x}{e^{j+1}\log R} e^{C(x)}~\text{for some $t$} \r)\]
and 
\[\P_2 = \sum_{k\le j \le \log (\frac{\log x}{\log R}) -B-2} \P\left( \prod_{\ell = j }^{\lfloor \log (\frac{\log x}{\log R}) \rfloor-B-2}
|I_{\ell}(\frac{1}{2}-\frac{k}{\log x} +i t(\ell)) |^{-1} >\frac{\log x}{e^{j+1}\log R} e^{C(x)}~\text{for some $t$} \r),\]
where $|t|\le 1/2$. 
We focus on bounding $\P_1$, and $\P_2$ can be estimated similarly. Replace the set of all $|t|\le 1/2$ by the discrete set 
\[\mathcal{T}(x, j): = \left\{\frac{n}{((\log x)/e^{j+1}) \log ((\log x)/e^{j+1}) }: |n|\le ((\log x)/e^{j+1}) \log ((\log x)/e^{j+1})  \r\},
\]
and apply the union bound to get 
\[\P_1 \le \sum_{\substack{k\le j \le \log (\frac{\log x}{\log R}) -B-2\\ t(j) \in \mathcal{T}(x,j)}} \P\left( \prod_{\ell = j }^{\lfloor \log (\frac{\log x}{\log R}) \rfloor-B-2}
|I_{\ell}(\frac{1}{2}-\frac{k}{\log x} +it(\ell)) | >\frac{\log x}{e^{j+1}\log R} e^{C(x)} \r).\]
By using Chebyshev's inequality this is at most 
\[
\le \sum_{\substack{k\le j \le \log (\frac{\log x}{\log R}) -B-2\\ t(j) \in \mathcal{T}(x,j)}} \frac{1}{(\frac{\log x}{e^{j+1}\log R} e^{C(x)})^{2}} \E[ \prod_{\ell = j }^{\lfloor \log (\frac{\log x}{\log R}) \rfloor-B-2}
|I_{\ell}(\frac{1}{2}-\frac{k}{\log x} +it(\ell)) |^{2}  ].
\]
Since $f(n)$ and $f(n)n^{it}$ have the same law, the above is 
\[\ll \sum_{\substack{k\le j \le \log (\frac{\log x}{\log R}) -B-2}} \frac{| \mathcal{T}(x,j)|}{(\frac{\log x}{e^{j+1}\log R} e^{C(x)})^{2}} \E[ \prod_{\ell = j }^{\lfloor \log (\frac{\log x}{\log R}) \rfloor-B-2}
|I_{\ell}(\frac{1}{2}-\frac{k}{\log x} ) |^{2}  ]. \]
The expectation here is, again through a mean square calculation \eqref{eqn: meansquare}, $\ll \frac{\log x}{e^{j+1}\log R}$. Note $|\mathcal{T}(x, j)| \le ((\log x)/e^{j+1}) \log ((\log x)/e^{j+1}) $. 
We conclude that 
\[ \P_1 \ll \sum_{\substack{k\le j \le \log (\frac{\log x}{\log R}) -B-2}}e^{\log \log R-2C(x) + \log \log ( \log x / e^{j+1})   } \ll e^{-C(x)}, \]
where in the last step we used that $C(x)=   \log \log R + 100 \log \log \log x$. Thus we complete the proof of Proposition~\ref{key prop 2}.

\section{Proof of Theorem~\ref{thm: lower}}\label{s: lowerthm}
 We first notice that if $R>x^{\frac{1}{A}}$ for any fixed large constant $A$, then $\CA_R(x)$ is a set of elements with only $O_A(1)$ number of prime factors. This would immediately imply that $\E[|\sum_{n\in \AR} f(n)|^{4}] \ll_A |\AR|^{2} $ and by \eqref{eqn: 4th}, the conclusion follows. From now on, we may assume that 
\begin{equation}\label{eqn: Rsquare}
R\le x^{\frac{1}{A}}.  
 \end{equation}
The proof strategy of Theorem~\ref{thm: lower} again follows from \cite{HarperLow}. The main differences lie in the design of the barrier events and taking advantage of $R$ being large.  In particular, we do not need a ``two-dimensional Girsanov-type" calculation which makes our proof less technical. 
We first do the reduction step to reduce the problem to understanding certain averages of random Euler products, as in the upper bound proof. 
\begin{proposition}\label{Prop: lower L2} 
There exists a large constant $C$ such that the following is true. Let $V$ be a sufficiently large
fixed constant. Let $x$ be sufficiently large and $\log \log R \gg \sqrt{\log \log x}$. Let $F^{(R)}(s)$ be defined as in \eqref{eqn: F}. Then,
uniformly for all $1/2 \le q\le 9/10$ and any large $V$, we have $\|\sum_{n\in \AR} f(n)  \|_{2q}$
\[ \gg \sqrt{\frac{x}{\log x}}  \left( \Big\|\int_{-\frac{1}{2}}^{\frac{1}{2}} | F^{(R)}(\frac{1}{2} +\frac{4V}{\log x} + it)|^{2} dt\Big\|_q^{\frac{1}{2}}  - \frac{C}{e^{V}} \Big\|\int_{-\frac{1}{2}}^{\frac{1}{2}} | F^{(R)}(\frac{1}{2} +\frac{2V}{\log x} + it)|^{2} dt\Big\|_q^{\frac{1}{2}} -C \right).   \]
\end{proposition}
The proof of Proposition~\ref{Prop: lower L2} is in Section~\ref{s: lowereuler}. 
The remaining tasks are to give a desired lower bound on $\|F^{(R)}(\frac{1}{2} +\frac{4V}{\log x} + it)\|_q^{\frac{1}{2}}$ and an upper bound on $\|F^{(R)}(\frac{1}{2} +\frac{2V}{\log x} + it)\|_q^{\frac{1}{2}}$. 
The upper bound part is simple. Indeed, simply apply H\"older's inequality and do a mean square calculation \eqref{eqn: meansquare} to get
\begin{equation}\label{eqn: plain}
\E[(\int_{-\frac{1}{2}}^{\frac{1}{2}}|F^{(R)}(\frac{1}{2} +\frac{2V}{\log x} + it)|^{2}  dt)^{q}] \ll \Big(\int_{-\frac{1}{2}}^{\frac{1}{2}}\E[|F^{(R)}(\frac{1}{2} +\frac{2V}{\log x} + it)|^{2}]  dt \Big)^{q} \ll  \Big(\frac{\log x}{V\log R}  \Big)^{q} . 
\end{equation}

We next focus on the main task, giving a good lower bound on $\|F^{(R)}(\frac{1}{2} +\frac{4V}{\log x} + it)\|_q^{\frac{1}{2}}$. For each $t\in \mathbb{R}$, 
we use $L(t)$ to denote the event that for all $\lfloor \log V \rfloor +3 \le j \le \log \log x - \log \log R -B -2$, the following holds
\begin{equation}\label{eqn: Lt}
  (\frac{\log x}{e^{j+1} \log R} e^{D(x)}  )^{-B} \le \prod_{\ell = j }^{\lfloor \log \log x -\log \log R \rfloor-B-2}
|I_{\ell}(\frac12+\frac{4V}{\log x} +it) | \le \frac{\log x}{e^{j+1}\log R} e^{D(x)},  
\end{equation}
where $D(x):= c\sqrt{\log \log x -\log \log R} $ with 
\begin{equation}\label{eqn: c}
    c= \frac{1}{4} \min\Big\{ \frac{\log \log R}{\sqrt{\log \log x-\log\log R} } , 1 \Big\} \asymp 1.
\end{equation}
We are now ready to define a random set 
\begin{equation}\label{eqn: L}
    \mathcal{L}: = \{-1/2\le t \le 1/2: L(t)~\text{defined by \eqref{eqn: Lt} holds}\}.
\end{equation}
It is clear that
\begin{equation}\label{eqn: onlyrandom}
\E[(\int_{-\frac{1}{2}}^{\frac{1}{2}}|F^{(R)}(\frac{1}{2} +\frac{4V}{\log x} + it)|^{2}  dt)^{q}] \ge \E[(\int_{\mathcal{L}}|F^{(R)}(\frac{1}{2} +\frac{4V}{\log x} + it)|^{2}  dt)^{q}].
\end{equation}
We use the following estimate and defer its proof to Section~\ref{s: lowerprop}.

\begin{proposition}\label{prop: random low}Let $V$ be a large
fixed constant. Let $x$ be sufficiently large and $\log \log R \gg \sqrt{\log \log x}$. Let $F^{(R)}(s)$ be defined as in \eqref{eqn: F}. Let $\mathcal{L}$ be the random set defined in \eqref{eqn: L}. Then uniformly for any $1/2\le q\le 9/10$, we have
\begin{equation}\label{eqn: lowbound}
\E[(\int_{\mathcal{L}}|F^{(R)}(\frac{1}{2} +\frac{4V}{\log x} + it)|^{2}  dt)^{q}] \gg \Big(\frac{\log x}{V\log R}  \Big)^{q} . 
\end{equation}
\end{proposition}
Plug \eqref{eqn: onlyrandom}, \eqref{eqn: lowbound} and \eqref{eqn: plain} into Proposition~\ref{Prop: lower L2} with $q=\frac{1}{2}$
(and choosing $V$ to be a sufficiently large fixed constant so that $C/e^{V}$ kills the implicit constant) to get that 
 \[\E[|\sum_{n\in \CA_R(x)} f(n) |] \gg  \sqrt{|\CA_{R}(x)|}, \]
 where we remind the readers the size of $|\CA_{R}(x)|$ is estimated in \eqref{eqn: sieve}. 
 This completes the proof of Theorem~\ref{thm: lower}.

\section{Proof of Proposition~\ref{Prop: lower L2}}\label{s: lowereuler}

The proof proceeds the same as in \cite[Proposition 3]{HarperLow} (see also \cite{HNR15}) and we provide a self-contained proof here and highlight some small modifications. 

Let $P(n)$ denote the largest prime factor of $n$ as before. We have assumed that \eqref{eqn: Rsquare} holds, e.g. $R\le \sqrt{x}$ (This restriction is not crucial but makes the notation later easier). Let $\epsilon$ denote a Rademacher random variable independent of $f(n)$, and recall that $\sumstar$ indicates that the variable $n$ under the summation is $R$ rough. For $1/2\le q\le 9/10$, we have 
\[
\begin{split}
    \E[|\sumstar_{\substack{n\le x\\ P(n)>\sqrt{x}}}f(n) |^{2q}] & = \frac{1}{2^{2q}} \E[|\sumstar_{\substack{n\le x\\ P(n)\le\sqrt{x}}}f(n) + \sumstar_{\substack{n\le x\\ P(n)>\sqrt{x}}}f(n)+\sumstar_{\substack{n\le x\\ P(n)>\sqrt{x}}}f(n)-\sumstar_{\substack{n\le x\\ P(n)\le\sqrt{x}}}f(n)|^{2q}]\\
    & \le \E[|\sumstar_{\substack{n\le x\\ P(n)\le\sqrt{x}}}f(n) + \sumstar_{\substack{n\le x\\ P(n)>\sqrt{x}}}f(n)|^{2q}] + \E[|\sumstar_{\substack{n\le x\\ P(n)>\sqrt{x}}}f(n)-\sumstar_{\substack{n\le x\\ P(n)\le\sqrt{x}}}f(n)|^{2q}]\\
    & = 2\E[|\epsilon \sumstar_{\substack{n\le x\\ P(n)>\sqrt{x}}}f(n) + \sumstar_{\substack{n\le x\\ P(n)\le\sqrt{x}}}f(n)|^{2q}] = 2\E[|\sumstar_{n\le x} f(n)|^{2q}],
\end{split}
\]
where the last step we used the law of \[\epsilon \sumstar_{\substack{n\le x\\ P(n)>\sqrt{x}}}f(n)= \epsilon \sum_{\sqrt{x}<p\le x} f(p) \sumstar_{m\le x/p} f(m) \] conditional on $(f(p))_{R\le p \le \sqrt{x}}$ is the same as the law of $\sumstar_{\substack{n\le  x\\ P(n)>\sqrt{x}}}f(n)$. By the above deduction, it suffices to give a lower bound on $\|\sumstar_{\substack{n\le x\\ P(n)>\sqrt{x}}}f(n)\|_{2q}$. 
Do the decomposition
\[\sumstar_{\substack{n\le x\\ P(n)>\sqrt{x}}}f(n) =\sum_{ \sqrt{x}\le p \le x} f(p) \sumstar_{m\le x/p}f(m).\]
The inner sum is determined by $(f(p))_{R\le p\le \sqrt{x}}$ and apply the Khintchine's inequality (see \cite[Lemma 3.8.1]{Gut} for the Rademacher case, and the Steinhaus case may be proved similarly) to get
\[\E[|\sumstar_{\substack{n\le x\\ P(n)>\sqrt{x}}}f(n)|^{2q}] \gg \E[(\sum_{\sqrt{x}< p \le x} |\sumstar_{m\le x/p}f(m)|^{2} )^{q}] \ge \frac{1}{(\log x)^{q}} \E[(\sum_{\sqrt{x}< p \le x} \log p\cdot|\sumstar_{m\le x/p}f(m)|^{2})^{q}]. \]
Next, do the smoothing step as we did in the upper bound case. Again set $X = e^{\sqrt{\log x}}$. 
Write 
\[\sum_{\sqrt{x}< p\le x} \log p \cdot |\sumstar_{m\le x/p}f(m)|^{2} = \sum_{\sqrt{x}<p\le x} \log p \cdot  \frac{X}{p} \int_{p}^{p(1+1/X)} |\sumstar_{m\le x/p} f(m)|^{2}dt. \]
One has $|a+b|^{2}\ge a^{2}/4 - \min\{|b|^{2}, |a/2|^{2}\}\ge 0$ and thus the above is at least 
\begin{equation}\label{eqn: sqrt}
    \begin{split}
   & \frac{1}{4} \sum_{\sqrt{x}<p\le x} \log p \cdot  \frac{X}{p} \int_{p}^{p(1+1/X)} |\sumstar_{m\le x/t} f(m)|^{2}dt \\
   & -\sum_{\sqrt{x}<p\le x} \log p \cdot \frac{X}{p} \int_{p}^{p(1+1/X)} \min\{|\sumstar_{x/t\le m\le x/p} f(m)|^{2}, \frac{1}{4} |\sumstar_{m\le x/t} f(m)|^{2}\}.
   \end{split}
\end{equation}
It follows that the quantity we are interested in has the lower bound
\begin{equation}\label{eqn: mid step}
    \begin{split}
\E[|\sumstar_{\substack{n\le x\\ P(n)>\sqrt{x}}}f(n)|^{2q}] \ge & \frac{1}{(\log x)^{q}}\E[(\frac{1}{4} \sum_{\sqrt{x}<p\le x} \log p \cdot \frac{X}{p} \int_{p}^{p(1+1/X)} |\sumstar_{m\le x/t} f(m)|^{2}dt)^{q}] \\
& - \frac{1}{(\log x)^{q}}\E[(\sum_{\sqrt{x}<p\le x} \log p \cdot \frac{X}{p} \int_{p}^{p(1+1/X)} |\sumstar_{ x/t<m\le x/p} f(m)|^{2}dt)^{q}].
\end{split}
\end{equation}
Use H\"older's inequality and throw away the $R$-rough condition to upper bound the subtracted term in \eqref{eqn: mid step} by
\[\begin{split}
    & \le \frac{1}{(\log x)^{q}} \left(\sum_{\sqrt{x}<p\le x} \log p \cdot \frac{X}{p} \int_{p}^{p(1+1/X)}\E[ |\sum_{ x/t<m\le x/p} f(m)|^{2}]dt\r)^{q}\\
    & \ll \frac{1}{(\log x)^{q}} \Big(\sum_{\sqrt{x}<p\le x } \log p \cdot  (\frac{x}{pX} +1) \Big)^{q} \ll \frac{1}{(\log x)^{q}}(\frac{x\log x}{X} +x)^{q} \ll (\frac{x}{\log x})^{q}. 
\end{split}
\]
The first term in \eqref{eqn: mid step} (without the factor $1/4(\log x)^{q}$) is 
\[
\begin{split}
& \E[(\sum_{\sqrt{x}<p\le x} \log p \cdot \frac{X}{p} \int_{p}^{p(1+1/X)} |\sumstar_{m\le x/t} f(m)|^{2}dt)^{q}] \\
   & \ge \E[(\int_{\sqrt{x}}^{x} \sumstar_{\frac{t}{1+1/X} <p\le t } \log p \cdot  \frac{X}{p}|\sumstar_{m\le x/t}f(m)|^{2} dt )^{q}]\\
   & \gg \E [(\int_{\sqrt{x}}^{x}|\sumstar_{m\le x/t}f(m)|^{2}dt )^{q}] = x^{q} \E[(\int_{1}^{\sqrt{x}  }  |\sumstar_{m\le z} f(m) |^{2} \frac{dz}{z^{2}})^{q}],
\end{split}
\]
where in the last inequality we used the prime number theorem.
To this end, we impose the smooth condition to invert the sums to Euler products. We have for any large $V$, 
\[
\begin{split}
 &   \E[(\int_{1}^{\sqrt{x}  }  |\sumstar_{m\le z} f(m) |^{2} \frac{dz}{z^{2}})^{q}] \ge \E[(\int_{1}^{\sqrt{x}  }  |\sumstar_{\substack{m\le z\\ x-\text{smooth}}} f(m) |^{2} \frac{dz}{z^{2+8V/\log x}})^{q}]\\
 & \ge \E[(\int_{1}^{+\infty  }  |\sumstar_{\substack{m\le z\\ x-\text{smooth}}} f(m) |^{2} \frac{dz}{z^{2+8V/\log x}})^{q}]- \E[(\int_{\sqrt{x}}^{+\infty  }  |\sumstar_{\substack{m\le z\\ x-\text{smooth}}} f(m) |^{2} \frac{dz}{z^{2+8V/\log x}})^{q}]\\
 & \ge \E[(\int_{1}^{+\infty  }  |\sumstar_{\substack{m\le z\\ x-\text{smooth}}} f(m) |^{2} \frac{dz}{z^{2+8V/\log x}})^{q}]- \frac{1}{e^{2Vq}}\E[(\int_{1}^{+\infty  }  |\sumstar_{\substack{m\le z\\ x-\text{smooth}}} f(m) |^{2} \frac{dz}{z^{2+4V/\log x}})^{q}].
\end{split}
\]
Apply Lemma~\ref{lem: pars} to get that the first term is 
\begin{equation}\label{eqn: 1stterm}
 \gg \E[(\int_{-\frac{1}{2}}^{\frac{1}{2}}|F^{(R)}(\frac{1}{2} + \frac{4V}{\log x} + it)|^{2} dt )^{q}].   
\end{equation}
For the second term, an application of Lemma~\ref{lem: pars} gives it is bounded by
\begin{equation}\label{eqn: 2ndterm}
    \ll e^{-2Vq} \E[(\int_{-\infty}^{+\infty} \frac{|F^{(R)}(\frac{1}{2} + \frac{2V}{\log x} + it)}{|\frac{1}{2} + \frac{2V}{\log x} + it|^{2}} dt )^{q}] \ll e^{-2Vq} \E[(\int_{-\frac{1}{2}}^{\frac{1}{2}} |F^{(R)}(\frac{1}{2} + \frac{2V}{\log x} + it)|^{2} )^{q}]
\end{equation}
where in the last step we used the fact that $f(n)n^{it}$ has the same law as $f(n)$ and $\sum_{n \ge 1 }n^{-2}$ converges. Bounds in \eqref{eqn: 1stterm} and \eqref{eqn: 2ndterm} together give the desired bound for the first term in \eqref{eqn: mid step} and we complete the proof.

\section{Proof of Proposition~\ref{prop: random low}}\label{s: lowerprop}
In this section, 
we prove Proposition~\ref{prop: random low}. The proof significantly relies on the following proposition, which is a mean value estimate of the product of $|F^{(R)}(\sigma + it_1)|^{2}$
and $|F^{(R)}(\sigma + it_2)|^{2}$. Our upper bound matches the guess if you pretend the two products are independent. 
\begin{proposition}\label{prop: correlation}
Let $V$ be a large
fixed constant. Let $x$ be sufficiently large and $\log \log R \gg \sqrt{\log \log x}$. Let $F^{(R)}(s)$ be defined as in \eqref{eqn: F}. Let $\mathcal{L}$ be the random set defined in \eqref{eqn: L}. Then we have
    \begin{equation}\label{eqn: simple}
\E[(\int_{\mathcal{L}}|F^{(R)}(\frac{1}{2} +\frac{4V}{\log x} + it)|^{2}  dt)^{2}]  \ll (\frac{\log x}{V\log R})^{2}. 
\end{equation}
\end{proposition}

\begin{proof}[Proof of Proposition~\ref{prop: random low} assuming Proposition~\ref{prop: correlation}]
The proof starts with an application of H\"older's inequality. We have 
\begin{equation}\label{eqn: smart}
\E[(\int_{\mathcal{L}}|F^{(R)}(\frac{1}{2} +\frac{4V}{\log x} + it)|^{2}  dt)^{q}] \ge \frac{(\E[\int_{\mathcal{L}}|F^{(R)}(\frac{1}{2} +\frac{4V}{\log x} + it)|^{2}  dt])^{2-q}}{(\E[(\int_{\mathcal{L}}|F^{(R)}(\frac{1}{2} +\frac{4V}{\log x} + it)|^{2}  dt)^{2}])^{1-q}}.
\end{equation}
Proposition~\ref{prop: correlation} gives a desired upper bound for the denominator. We next give a lower bound on the numerator. By using that $f(n)n^{it}$ has the same law as $f(n)$, the numerator is
\[(\int_{-1/2}^{1/2}\E [\mathbf{1}_{L(t)}|F^{(R)}(\frac{1}{2} +\frac{4V}{\log x} + it)|^{2}]  dt)^{2-q}= (\E[ \mathbf{1}_{L(0)}|F^{(R)}(\frac{1}{2} +\frac{4V}{\log x} )|^{2}] )^{2-q} .\]
We next use Proposition~\ref{prop 5} by taking $n=\lfloor \log \log x - \log \log R \rfloor - (B+1) - \lfloor \log V \rfloor$, $a =D(x)=c\sqrt{\log \log x -\log \log R } $ and $h(j)=0$ to conclude that $\tilde{\P}(\mathbf{1}_{L(0)})\gg 1$. Combining with the mean square calculation~\eqref{eqn: meansquare}, we have
\begin{equation}\label{eqn: L0}
    \E [\mathbf{1}_{L(0)}|F^{(R)}(\frac{1}{2} +\frac{4V}{\log x} + it)|^{2}] \gg \tilde{\P}(\mathbf{1}_{L(0)})\cdot \E [|F^{(R)}(\frac{1}{2} +\frac{4V}{\log x} + it)|^{2}] \gg \frac{\log x}{V \log R}.
\end{equation}
We complete the proof by plugging \eqref{eqn: simple} and \eqref{eqn: L0} into \eqref{eqn: smart}. 
\end{proof}

The proof of Proposition~\ref{prop: correlation} is a bit involved and its proof is inspired by \cite[Key proposition 5]{HarperLow} and \cite[Multiplicative chaos results 4]{Harperlargevalue}. We are not using the ``two-dimensional Girsanov-type" computation as used in \cite[Key proposition 5]{HarperLow} which significantly simplified the proof. We do not expect any further savings when $R$ is as large as stated in Proposition~\ref{prop: correlation} while for a smaller $R$, one might expect there could be further cancellation as in \cite[Key Proposition 5]{HarperLow} which may be verified by adapting the ``two-dimensional Girsanov-type" calculation.

\begin{proof}[Proof of Proposition~\ref{prop: correlation}]
    Expand the square and the left hand side of \eqref{eqn: simple} equals
    \[\E [ \int_{-1/2}^{1/2}\mathbf{1}_{L(t_1)} |F^{(R)}(\frac{1}{2} +\frac{4V}{\log x} + it_1)|^{2} dt_1 \int_{-1/2
    }^{1/2} \mathbf{1}_{L(t_2)}|F^{(R)}(\frac{1}{2} +\frac{4V}{\log x} + it_2)|^{2} dt_2   ]. \]
   By using that $f(n)n^{it}$ has the same law as $f(n)$, we write the above as ($t:= t_1-t_2$)
   \begin{equation}\label{eqn: tshift}
      \int_{-1}^{1}\E [\mathbf{1}_{L(0)} |F^{(R)}(\frac{1}{2} +\frac{4V}{\log x} )|^{2}  \mathbf{1}_{L(t)}|F^{(R)}(\frac{1}{2} +\frac{4V}{\log x} + it)|^{2} ] dt. 
   \end{equation}
For $|t|$ large enough, the two factors behave independently, which is the easier case. Indeed,  if $|t|>1/\log R$, drop the indicator functions and bound the corresponding integration by
\[\ll \max_{ 1/\log R < |t|\le 1} \E [  |F^{(R)}(\frac{1}{2} +\frac{4V}{\log x})|^{2} \cdot |F^{(R)}(\frac{1}{2} +\frac{4V}{\log x} + it)|^{2}   ]  .\] 
 Apply the two dimensional mean square calculation \eqref{eqn: 4.3} with $(x, y)=(R, x) $  to conclude that the above is 
 \[\ll \Big(\frac{\log x}{V \log R}\Big)^{2}.  \]
 
We next focus on the case $|t|\le 1/\log R$. Since $f(p)$ are independent of each other, we can decompose the Euler products into pieces and analyze their contributions to \eqref{eqn: tshift} separately.
Define the following three sets of primes based on the sizes of primes
\[\mathcal{P}_1: = \{p~\text{prime}: R\le p < x^{e^{-(\lfloor \log \log x -\log \log R \rfloor-B-2)}}\},\]
\[\mathcal{P}_2: = \{p~\text{prime}:  x^{e^{-(\lfloor \log \log x -\log \log R \rfloor-B-2)}} \le p \le x^{e^{-(\lfloor \log V \rfloor +3)}}\},\]
and 
\[\mathcal{P}_3: = \{p~\text{prime}: x^{e^{-(\lfloor \log V \rfloor +3)}} < p \le x \}.\]
We proceed as follows. Note that the events $L(0)$ and $L(t)$ are irrelevant to $f(p)$ for $p\in \mathcal{P}_1 \cup \mathcal{P}_3$. For partial products over primes $p\in \mathcal{P}_1 \cup \mathcal{P}_3$, we directly do mean square calculations.
For partial products over primes $p\in \mathcal{P}_2$, we will crucially use the indicator functions $\mathbf{1}_{L(0)}$ and $\mathbf{1}_{L(t)}$ defined in \eqref{eqn: Lt} with $j= \lfloor \log V \rfloor +3$. This separation gives that the integration in \eqref{eqn: tshift} over $|t|\le 1/\log R$ is
\begin{equation}\label{eqn: rewrite}
 \begin{split}
    \int_{|t|\le \frac{1}{\log R}} \E[\prod_{p\in \mathcal{P}_1 \cup \mathcal{P}_3 } |1-\frac{f(p)}{p^{\frac{1}{2} +\frac{4V}{\log x}}}|^{-2} |1-\frac{f(p)}{p^{\frac{1}{2} +\frac{4V}{\log x}+it}}|^{-2}] \\ \times \E[\mathbf{1}_{L(0)} \mathbf{1}_{L(t)} \prod_{p\in \mathcal{P}_2 } |1-\frac{f(p)}{p^{\frac{1}{2} +\frac{4V}{\log x}}}|^{-2} |1-\frac{f(p)}{p^{\frac{1}{2} +\frac{4V}{\log x}+it}}|^{-2}]  dt. 
\end{split}   
\end{equation}
We first upper bound the expectation over primes in $\mathcal{P}_1 \cup \mathcal{P}_3$ uniformly over all $t$. By using independence between $f(p)$ and \eqref{eqn: 4.2}, we can bound it as 
\begin{equation}\label{eqn: p13}
  \ll \exp\Big( \sum_{p\in \mathcal{P}_1} \frac{4}{p^{1+\frac{8V}{\log x}}} + \sum_{p\in \mathcal{P}_3} \frac{4}{p^{1+\frac{8V}{\log x}}} \Big).  
\end{equation}
By simply using the prime number theorem and the definition of $\mathcal{P}_1$ and $\mathcal{P}_3$, one has that both sums in \eqref{eqn: p13}  are $\ll 1$ so that \eqref{eqn: p13} is $\ll 1$, where we remind readers that $B$ is a fixed constant. Now our task is reduced to establishing the following
\begin{equation}\label{eqn: p2}
    \int_{|t|\le \frac{1}{\log R}}   \E[\mathbf{1}_{L(0)} \mathbf{1}_{L(t)} \prod_{p\in \mathcal{P}_2 } |1-\frac{f(p)}{p^{\frac{1}{2} +\frac{4V}{\log x}}}|^{-2} |1-\frac{f(p)}{p^{\frac{1}{2} +\frac{4V}{\log x}+it}}|^{-2}]  dt \ll \Big(\frac{\log x}{V \log R}\Big)^{2}.  
\end{equation}

Our strategy would be, roughly speaking, using the barrier event $\mathbf{1}_{L(t)}$ to bound certain partial products involved with $t$ directly and then use the mean square calculation to deal with the rest of the products. The exact partial products that we will apply barrier events would depend on the size of $t$.  

We first do a simple case, which helps us get rid of the very small $ t$, say $|t|<V/\log x$. We use the condition $\mathbf{1}_{L(t)}$ and pull out the factors related to $L(t)$ to get that the contribution from $|t|<V/\log x$ is at most 
\[ 
\begin{split}
   & \ll  \int_{|t|\le \frac{V}{\log x}} e^{2c\sqrt{\log \log x- \log \log R}} \cdot (\frac{\log x}{V \log R})^{2}\cdot\E[\mathbf{1}_{L(0)}\prod_{p\in \mathcal{P}_2 } |1-\frac{f(p)}{p^{\frac{1}{2} +\frac{4V}{\log x}}}|^{-2}]  dt \\ & \ll \frac{V}{\log x} \cdot  e^{2c\sqrt{\log \log x- \log \log R}} \cdot \Big(\frac{\log x}{V \log R}\Big)^{2} \cdot \E[\prod_{p\in \mathcal{P}_2 } |1-\frac{f(p)}{p^{\frac{1}{2} +\frac{4V}{\log x}}}|^{-2}] \\
   & \ll \Big(\frac{\log x}{V \log R}\Big)^{2},
\end{split}
\]
where in the second to last step we dropped the $\mathbf{1}_{L(0)}$ condition, and in the last step we applied \eqref{eqn: meansquare} together with $ \log R \ge \exp(4c  \sqrt{\log \log x}) $ where $c$ is defined in \eqref{eqn: c}. Thus we only need to establish the following 
\begin{equation}\label{eqn: p2t}
    \int_{\frac{V}{\log x} \le |t|\le \frac{1}{\log R}}   \E[\mathbf{1}_{L(0)} \mathbf{1}_{L(t)} \prod_{p\in \mathcal{P}_2 } |1-\frac{f(p)}{p^{\frac{1}{2} +\frac{4V}{\log x}}}|^{-2} |1-\frac{f(p)}{p^{\frac{1}{2} +\frac{4V}{\log x}+it}}|^{-2}]  dt \ll \Big(\frac{\log x}{V \log R}\Big)^{2} .  
\end{equation}

We now enter the crucial part where we will apply the barrier events according to the size of $|t|$.
We \textit{decompose the set $\mathcal{P}_2$ into two parts according to $|t|$}.
For each fixed $V/\log x\le |t| \le 1/\log R$, we write 
\[\mathcal{P}_2 = \mathcal{S}(t) \cup \mathcal{M}(t), \]
where 
\[\mathcal{S}(t):=\{p~\text{prime}:  x^{e^{-(\lfloor \log \log x -\log \log R \rfloor-B-2)}} \le p \le e^{\frac{V}{|t|}}\},\]
and 
\[\mathcal{M}(t):= \{p~\text{prime}:  e^{\frac{V}{|t|}} \le p \le x^{e^{-(\lfloor \log V \rfloor +3)}}\}.\]
The set of primes $\mathcal{S}(t)$ would be those we will apply barrier events and $\mathcal{M}(t)$ would be estimated by a mean square calculation. Note that for $p\in \mathcal{M}(t)$, there is a nice decorrelation as we needed in \eqref{eqn: 4.3} due to that $p\ge e^{V/|t|}$. 
Let us now see how such a decomposition of $\mathcal{P}_2$ would help us. We use a local notation
\[G(p, t): = |1-\frac{f(p)}{p^{\frac{1}{2} + \frac{4V}{\log x}+it}}|^{-2}. \]
Then the quantity in \eqref{eqn: p2t} is the same as 
\[\int_{\frac{V}{\log x} \le |t|\le \frac{1}{\log R}} \E[\mathbf{1}_{L(0)} \mathbf{1}_{L(t)} \prod_{p\in \mathcal{P}_2 }G(p, 0) \prod_{p\in \mathcal{S}(t) }G(p, t) \prod_{p\in \mathcal{M}(t) }G(p, t)  ]  dt.\]
We apply the barrier events condition $\mathbf{1}_{L(t)}$ to bound the product over $p\in \mathcal{S}(t)$ so that the above is at most 
\begin{equation}\label{eqn: 8.6}
  \ll  \Big(\frac{V}{ \log R}\Big)^{2} \cdot e^{2c\sqrt{\log \log x - \log \log R}} \cdot \int_{\frac{V}{\log x} \le |t|\le \frac{1}{\log R}}  \frac{1}{t^{2}}\E[\mathbf{1}_{L(0)} \prod_{p\in \mathcal{P}_2 }G(p, 0) \prod_{p\in \mathcal{M}(t) }G(p, t)  ]  dt. 
\end{equation}
We next upper bound the expectation in \eqref{eqn: 8.6} uniformly for all $V/\log x\le |t| \le 1/\log R$. We first drop the indicator function and rewrite the product based on the independence between $f(p)$ to derive that
\[\E[\mathbf{1}_{L(0)} \prod_{p\in \mathcal{P}_2 }G(p, 0) \prod_{p\in \mathcal{M}(t) }G(p, t)]\le \E[ \prod_{p\in \mathcal{S}(t) }G(p, 0)] \cdot  \E [\prod_{p\in \mathcal{M}(t) }G(p, 0)G(p, t)].\]
 Use the mean square calculation results in \eqref{eqn: meansquare} and \eqref{eqn: 4.3} to further get an upper bound on the expectation
 \[\ll \frac{V/|t|}{\log R} \cdot \Big(\frac{t\log x}{V^{2}}\Big)^{2}\ll \frac{|t|(\log x)^{2}}{V^{3}\log R} .   \]
Now we plug the above bound to \eqref{eqn: 8.6} to get that \eqref{eqn: 8.6} is crudely bounded by 
\[\Big(\frac{\log x}{ \log R}\Big)^{2} \cdot \frac{e^{2c\sqrt{\log \log x- \log \log R}}}{V\log R} \cdot \int_{\frac{V}{\log x} \le |t|\le \frac{1}{\log R}} \frac{1}{|t|} dt \ll \Big(\frac{\log x}{ \log R}\Big)^{2} \cdot \frac{\log \log x}{V\cdot e^{2c\sqrt{\log\log x}}},  \]
where we used that $ \log R \ge \exp(4c  \sqrt{\log \log x}) $ and $c\gg 1$ is defined in \eqref{eqn: c}. The last factor tends to zero as $x\to+\infty$, so it is surely $\ll \frac{1}{V^{2}}$ for a fixed large constant $V$. This completes the proof of \eqref{eqn: p2t} and thus the proof of the proposition. 
\end{proof}

\section{Concluding remarks}\label{s: conclude}

\subsection{Typical behavior and small perturbations}\label{Sec: typical}
We give a sketch of the situation when $a(n)$ itself is independently and randomly chosen. 
We write 
\begin{equation}
   a(n) = r(n) X(n) 
\end{equation}
where $r(n)>0$ is deterministic and $X(n)$ are independently distributed with $\E[|X(n)|^{2}]=1$. We may naturally assume that there is some $r$ such that
\[r(n) \asymp r(m) \asymp r \]
for all $n, m$, i.e. no particular random variable would dominate the whole sum in size. One may also just assume $r=1$ throughout the discussion here. 
We claim that for typical $X(n)$, the random sums satisfy the sufficient condition established in \cite[Theorem 3.1]{SoundXu} on having a Gaussian limiting distribution.

The key condition one needs to verify is that almost surely (in terms of over $X(n)$), we have 
\begin{equation}\label{eqn: R}
R_N(\textbf{a}) : =\sum_{\substack{ m_i, n_j\le N \\
 m_i\neq n_j \\ m_1m_2=n_1n_2 }}  a(n_1)a(n_2) \overline{a(m_1) } \overline{a(m_2)} = o(r^{4}N^{2}).    
\end{equation}
The proof of \eqref{eqn: R} is straightforward. By using the divisor bound, we know there are $\ll N^{2+\epsilon}$ number of quadruples $(m_1, m_2, n_1, n_2)$ under the summation. If we expect some square-root cancellation among $a(n_1)a(n_2) \overline{a(m_1) } \overline{a(m_2)}$, then $R_N(\textbf{a})$ above should be around $r^{4}N^{1+\ee}$ typically. 
Indeed, by using the fact that all $a(n)$ are independent, we have the $L^{2}$ bound 
\[\E[|R_N|^{2}] = \E[R_N \overline{R_N}] \ll r^{8} N^{2+\ee}. \]
This leads to, 
almost surely (in terms of over $X(n)$), that we have 
\[R_N(\textbf{a}) = o(r^{4}N^{2}). \]
To this end, by using \cite[Theorem 3.1]{SoundXu}, almost surely, we have a central limit theorem for the random partial sums of a Steinhaus random multiplicative function.
See \cite[Theorem 1.2]{BNR} for a closely related result where they used the method of moments.

In Question~\ref{question main}, we asked if it is possible to characterize the choices of $a(n)$ that give better than square-root cancellation. On one hand, as discussed above, we know for typical $a(n)$, there is just square-root cancellation. On the other hand, if $a(n)$ is a deterministic multiplicative function taking values on the unit circle, then by the fact that $a(n)f(n)$ has the same distribution as $f(n)$ and the result established by Harper \eqref{eqn: harper}, the partial sums $\sum_{n\le N} a(n)f(n)$ have better than square-root cancellation. Our main theorems study one particular example of multiplicative nature.  Combining these observations, 
we believe that \textit{any small perturbation coming from $a(n)$ that destroys the multiplicative structure would make the better than square-root cancellation in \eqref{eqn: harper} disappear}. We ask the following question in a vague way as a sub-question of Question~\ref{question main}.
\begin{question}
Is it true that the only ``essential choice" of $a(n)$ leading to better than square-root cancellation is of multiplicative nature? 
\end{question}

\subsection{Threshold in other settings and the limiting distribution}
The main theorems of this paper prove that there is square-root cancellation for $\log \log R \gg (\log \log x)^{\frac{1}{2}}$. What is the limiting distribution then? We have remarked earlier that one may establish a central limit theorem when $R\gg \exp((\log x)^{c})$ for some constant $c<1$ by understanding the corresponding multiplicative energy. It becomes less clear for smaller $R$. 

\begin{question}
 What is the limiting distribution of  $\sum_{n\in \CA_R(x)} f(n)$ with ``proper" normalization, for all ranges of $R$?   
\end{question}

We finally comment that there is another family of partial sums that naturally has the threshold behavior for better than square-root cancellation. Let $\CA= [x, y]$ with $y\le x$. We would like to know for what range of $y$, typically, 
\[\sum_{x\le n \le x+y} f(n) = o(\sqrt{y}).\]
We believe one can adapt the argument here to find that the threshold behavior is around $\log (x/y) \approx  \sqrt{\log \log x}$. It is certainly interesting to understand the limiting distribution for the short interval case thoroughly, beyond the previous result in \cite{SoundXu}.

\bibliographystyle{plain}
\bibliography{BTSC}{}
\end{document}